\documentclass[11pt,a4paper]{article}

\usepackage{etoolbox}
\usepackage[
  top=1in,
  bottom=1in,
  left=0.75in,
  right=0.75in
]{geometry}
\usepackage[T1]{fontenc}
\usepackage[utf8]{inputenc}
\usepackage{lmodern}

\usepackage{amsmath,amssymb,amsthm,mathtools}
\usepackage{enumitem}
\usepackage{hyperref}
\usepackage[nameinlink,noabbrev]{cleveref}
\usepackage{float}

\usepackage{tikz}
\usepackage{tikz-cd}

\newtheoremstyle{mytheorem}
  {12pt}    
  {0pt}    
  {\itshape} 
  {}       
  {\bfseries} 
  {.}      
  {0.5em}  
  {}       

\theoremstyle{mytheorem}
\newtheorem{theorem}{Theorem}[section]

\newtheorem{corollary}[theorem]{Corollary}

\theoremstyle{definition}
\newtheorem{definition}[theorem]{Definition}
\newtheorem{example}[theorem]{Example}

\newcommand{\cluster}[1]{%
\tikz[baseline=(X.base)]{
\node[
    draw,
    circle,
    inner sep=1pt,
    minimum size=1.8em
] (X) {\strut #1};
}}

\title{The Exact Degree Threshold for Complement Connectivity}
\author{
Arsen Berk Tekdaş\\
\small Department of Mathematics, Bilkent University\\
\small 06800 Ankara, Türkiye\\
\small \texttt{berk.tekdas@bilkent.edu.tr}
}
\date{}

\BeforeBeginEnvironment{proof}{\vspace{-0.6em}}

\begin{document}

\maketitle

\begin{abstract}
\noindent
For a graph $G$, we investigate sufficient degree conditions under which
either $G$ or its complement $\overline{G}$ is $k$-connected. Motivated
by a theorem of Hellwig and Volkmann, we introduce the notion of
complement connectivity admissible pairs and define the parameter
$\mathcal{C}(k)$ as the minimum degree threshold guaranteeing that either
$G$ or $\overline{G}$ is $k$-connected whenever $\delta(G)\ge\mathcal{C}(k)$ and
$\Delta(G)\le |V(G)|-1-\mathcal{C}(k)$.
We prove that $\mathcal{C}(k)=2k-2$ for every $k\ge2$, thereby
determining the exact degree threshold. As an application, we derive
a Nordhaus--Gaddum type product condition, showing that for
$n\ge4k-5$, the condition $\delta(G)\delta(\overline{G})
\ge (2k-3)(n-2k+2)$ guarantees that either $G$ or $\overline{G}$ is $k$-connected.
Finally, we show that no analogous product condition depending only on $k$ can exist.
\end{abstract}

\section{Introduction}

Throughout this paper, all graphs are finite and simple. For a graph $G$, we write $\overline{G}$ for its complement, $\delta(G)$ and $\Delta(G)$ for its minimum and maximum degree, respectively, and $\kappa(G)$ for its connectivity. We follow the graph-theoretic terminology and notation of Diestel~\cite{Diestel} unless otherwise specified.

The relationship between a graph and its complement has been a central topic in graph theory for several decades. Beginning with the classical work of Nordhaus and Gaddum~\cite{Nordhaus-Gaddum}, numerous results have been established relating graph parameters of a graph and its complement. One of the early works devoted specifically to connectivity is due to Akiyama and Harary~\cite{Akiyama-Harary}, who investigated graphs for which both a graph and its complement satisfy prescribed graph-theoretic properties, including connectivity and $2$-connectivity. A classical result states that either a graph or its complement is connected~\cite{Bondy-Murty}. Hellwig and Volkmann \cite{Hellwig-Volkmann} further investigated the connectivity of a graph and its complement when both graphs are connected and proved that 
\[\kappa(G)+\kappa(\overline{G}) \ge \min\{\delta(G),\delta(\overline{G})\}+1\]

Motivated by these results, we investigate sufficient degree conditions guaranteeing that either a graph or its complement is $k$-connected. As a first step, we prove that if $G$ is a graph on $n\ge4k-1$ vertices satisfying
\[
\delta(G)\ge2k-1
\quad\text{and}\quad
\Delta(G)\le n-2k,
\]
then either $G$ or $\overline{G}$ is $k$-connected. Note that an equivalent condition is
\[
\delta(G)\ge 2k-1
\quad\text{and}\quad
\delta(\overline{G})\ge 2k-1,
\]
and we may use these conditions interchangeably throughout the paper.
We then introduce the following notion to capture this question in a more general setting. 

\begin{definition} 
Let $G=(V,E)$ be a graph with $|V|=n$. Then, a pair $(f,g)$ of functions $f,g:\mathbb{N} \to \mathbb{N}$ is called a \textit{complement connectivity admissible pair} if for any $k \ge 2$, $\delta(G) \ge f(k)$ and $\Delta(G) \le n-1-g(k)$ implies that either $G$ or $\overline{G}$ is $k$-connected. The set of all complement connectivity admissible pairs is denoted by $\text{CCAP}$.
\end{definition} 

Since we are concerned only with the symmetric case $f=g$, we define
$$\mathcal{C}(k) = \min_{(f,f) \in \text{CCAP}}\big\{f(k) \mid \text{$\delta(G)\ge f(k)$ and $\Delta(G) \le n-1-f(k) \implies G$ or $\overline{G}$ is $k$-connected}\big\}$$

Then, as an intermediate step, we determine the exact value of $\mathcal{C}(2)$ before establishing the following result, which is the main result of the paper.

\begin{theorem} \label{C(k)_main} $\mathcal{C}(k)=2k-2$ for any $k\ge 2$.
\end{theorem}

Finally, as an application of Theorem~\ref{C(k)_main}, we derive a
Nordhaus--Gaddum type product condition. More precisely, we show that
for $n\ge4k-5$, the condition $\delta(G)\delta(\overline{G})
\ge (2k-3)(n-2k+2)$
guarantees that either $G$ or $\overline{G}$ is $k$-connected.
We also show that no analogous product condition depending only on
$k$ can exist by constructing graphs for which
$\delta(G)\delta(\overline{G})$ is arbitrarily large while neither
$G$ nor $\overline{G}$ is $k$-connected.

The remainder of the paper is organized as follows. In
Section~2, we establish preliminary results that motivate and prepare the proof of the main theorem. Section~3 is devoted to the proof of Theorem~\ref{C(k)_main}, including the extremal construction that establishes the lower bound. In Section~4, we apply the main theorem to obtain a Nordhaus--Gaddum type product condition and show that no analogous condition depending only on $k$ can guarantee that either $G$ or $\overline{G}$ is $k$-connected.

\section{Preliminary Results}

In this section, we establish an upper bound for $\mathcal{C}(k)$, which will later be used to restrict the possible values of $\mathcal{C}(k)$. We then show, by determining $\mathcal{C}(2)$, that this bound is not tight in general.

\begin{theorem} \label{C(k)_upperbound} $\mathcal{C}(k) \le 2k-1$ for any $k\ge 2$.
\end{theorem}

\begin{proof}
Suppose that $G$ is not $k$-connected. Then, there exists $S\subset V$ with $|S|\le k-1$ such that $G-S$ is disconnected. Let $C_1,\dots, C_r$ be the components of $G-S$, where $r\ge 2$. Here, we claim that $|C_i| \ge k+1$ for each $i$. To prove this claim, let $x\in V(C_i)$. As $C_i$ is a component of $G-S$, neighbors of $x$ must lie in $V(C_i) \cup S$. Hence 
\[
d_G(x) \le |S|+|C_i|-1.
\]
We also know that $d_G(x) \ge 2k-1$, so 
\[
|S|+|C_i|-1 \ge 2k-1.
\]
Together with $|S|\le k-1$, we establish the claim that $|C_i| \ge k+1$.

Now, let $T \subset V$ such that $|T|\le k-1$. Since each $|C_i| \ge k+1$ in $G-S$ and there are no edges between any pairs of $C_i$'s, the induced subgraph
\[
H:=\overline{G}\left[\bigcup_{i=1}^r C_i\right]
\]
of $\overline{G}$ contains a spanning complete $r$-partite graph with parts
$C_1,\ldots,C_r$. As $|T|\le k-1$ and $|C_i|\ge k+1$ for each $i$, $H-T$ is still complete $r$-partite with no empty partitions and hence is connected. So, it remains to show that every vertex of $S\setminus T$ has a neighbor in $H-T$. Let $s \in S\setminus T$. Then, as $|S|,|T|\le k-1$ and $d_G(s) \le n-2k$, we have 
\[
d_{\overline{G}}(s) = n-1-d_G(s) \ge 2k-1 > |S\cup T|.
\]
So, $s$ must be adjacent to a vertex in $V\setminus (S\cup T) = V(H-T)$ in $\overline{G}$. Therefore, $\overline{G}-T$ is connected and thus $\overline{G}$ is $k$-connected.
\end{proof}

Note that these bounds are not tight even for $k=2$ as the following theorem shows.

\begin{theorem} \label{thm_k2}  $\mathcal{C}(2)=2$.
\end{theorem}

\begin{proof}
Suppose that $G$ is not $2$-connected. Then, we shall show that any two vertices in $\overline{G}$ are connected by at least two independent paths. Then the result follows by Menger's Theorem~\cite{Menger}. 

First, consider the case where $G$ is disconnected and let $C_1$ and $C_2$ be two components of $G$. Since $\delta(G) \ge 2$, we have $|V(C_1)|,\,|V(C_2)| \ge 3$. Here, let $u$ and $v$ be two arbitrary vertices of $G$.
\begin{itemize}
    \item If $u,v \in V(C_1)$, then we can choose two vertices $x,y \in V(C_2)$. Then $uxv$ and $uyv$ are two independent $u$-$v$ paths in $\overline{G}$.
    \item If $u\in V(C_1)$, $v\in V(C_2)$, taking $x\in V(C_1)$ and $y\in V(C_2)$, $uv$ and $uyxv$ are two independent $u$-$v$ paths in $\overline{G}$.
\end{itemize}

Next, consider the case where $\kappa(G) = 1$, that is, $G$ has a cut vertex, say $s$. Let $C_1,\dots,C_r$ be the components of $G-s$. Then, each component has at least $2$ vertices because $\delta(G)\ge 2$. Then, between any pair of vertices that does not contain $s$, there are at least two independent paths in $\overline{G}$ by the previous case. So, it is enough to show that there are at least two independent $s$-$t$ paths in $\overline{G}$ for an arbitrary vertex $t \in V(G)$. Without loss of generality, let $t \in V(C_1)$.
\begin{itemize}
    \item If $t$ is not adjacent to $s$ in $G$, then there must be another vertex $u$ that is not adjacent to $s$ in $G$ as $\Delta(G) \le n-3$. Here, if $u$ is not adjacent to $t$ in $G$, then $st$ and $sut$ are two independent $s$-$t$ paths in $\overline{G}$. Otherwise, since $t \in V(C_1)$, we have $u\in V(C_1)$. Then, taking $x\in V(C_2)$, $st$ and $suxt$ are two independent $s$-$t$ paths in $\overline{G}$.
    \item If $t$ is adjacent to $s$ in $G$, then there are two distinct vertices $x$ and $y$ that are not adjacent to $s$ in $G$, since $\Delta(G)\le n-3$. If $x,y\in V(C_1)$, then taking two distinct vertices $z,w$ from some component $C_j$ with $j\ne1$, $sxzt$ and $sywt$ are two independent $s$-$t$ paths in $\overline{G}$. If $x,y\notin V(C_1)$, then $sxt$ and $syt$ are two independent $s$-$t$ paths in $\overline{G}$. Finally, if exactly one of $x$ and $y$ belongs to $C_1$, say $x\in V(C_1)$ and $y\in V(C_j)$ for some $j\ne1$, then taking $z\in V(C_j)\setminus\{y\}$, $sxzt$ and $syt$ are two independent $s$-$t$ paths in $\overline{G}$.
\end{itemize} 
Since it is clear that $\mathcal{C}(2) > 1$, the result follows.
\end{proof}

\section{The Main Result}

In this section, we prove Theorem~\ref{C(k)_main}. Before giving the proof, we introduce the following convention for the figures.

Let $X=\{x_1,\dots,x_n\}$ and $Y=\{y_1,\dots,y_n\}$ be vertex sets. Throughout the figures, the symbol $\cluster{X}$ denotes the cluster $X$. A loop incident with $\cluster{X}$ indicates that the subgraph induced by $X$ is complete. An edge between $\cluster{X}$ and $\cluster{Y}$ indicates that every edge in $E(X,Y)$ is present.

\label{C(k)_main_proof}
\begin{proof}[Proof of Theorem~\ref{C(k)_main}] Given any $k\ge 2$, first we shall show that $\mathcal{C}(k) \ge 2k-2$ by constructing a graph $G=(V,E)$ with $\delta(G) \ge 2k-3$ and $\Delta(G) \le n-2k+2$ such that neither $G$ nor $\overline{G}$ is $k$-connected.

Let $V=A\sqcup B\sqcup C\sqcup D$ where $A=\{a_1,\dots,a_{k-1}\}$, $B=\{b_1,\dots,b_{k-1}\}$, $C=\{c_1,\dots,c_{k-1}\}$, $D=\{d_1,\dots,d_{k-1}\}$ and 
\[
E(G)=\{b_ib_j \mid 1\le i<j\le k-1\} \cup \{d_id_j \mid 1\le i<j\le k-1\} \cup E(A,B) \cup E(A,C) \cup E(C,D)
\]

Then
\[
E(\overline{G}) = \{a_ia_j \mid 1\le i<j\le k-1\} \cup \{c_ic_j \mid 1\le i<j\le k-1\} \cup E(A,D) \cup E(B,C) \cup E(B,D)
\]
One can see the construction of $G$ and $\overline{G}$ in the figure below.

\begin{figure}[htbp]
    \centering
    \includegraphics[width=15cm]{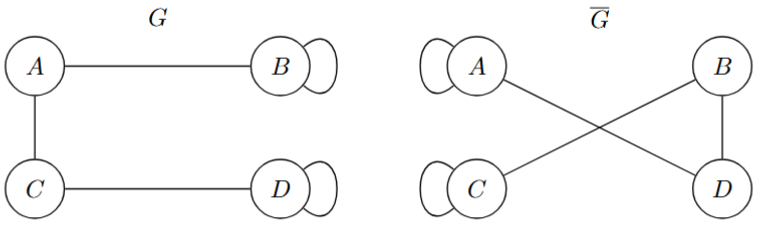}
    \caption{The construction of $G$ and $\overline{G}$ establishing $\mathcal{C}(k) \ge 2k-2$}
\end{figure}

Here, $\delta(G)=2k-3$, $\Delta(G)=2k-2$ and it follows immediately from the construction that $G$ is self-complementary, that is $G\cong \overline{G}$, with the isomorphism given by $a_i \mapsto d_i$, $b_i \mapsto a_i$, $c_i \mapsto b_i$, $d_i \mapsto c_i$, for any $1\le i\le k-1$. Now, since $G-A$ is disconnected with $|A|=k-1$, $G$ is not $k$-connected. As $G$ is self-complementary, $\overline{G}$ is not $k$-connected.

It remains to prove the reverse inequality $\mathcal{C}(k)\le 2k-2$. Suppose for a contradiction that there exists a graph $G$ on $n$ vertices satisfying $\delta(G)\ge 2k-2$ and $\Delta(G)\le n-2k+1$ such that neither $G$ nor $\overline{G}$ is $k$-connected. We consider two cases.

First, consider the case where $G$ or $\overline{G}$ is disconnected. Without loss of generality, assume that $G$ is disconnected. Let $C_1,\dots,C_r$ be its components. Then, in $\overline{G}$, every vertex of $C_i$ is adjacent to every vertex of $C_j$ when $i\ne j$. Hence, $\overline{G}$ contains a complete $r$-partite graph $K$ as a spanning subgraph. Now, let $v\in C_i$. Then, since $\delta(G) \ge 2k-2$ implies that $\Delta(\overline{G}) \le n-2k+1$, we have
\begin{equation}
    d_{\overline{G}}(v) \le n-2k+1
\end{equation}
and since $v$ can be adjacent to only the vertices of $C_i$ in $G$, we have $n-1-d_{\overline{G}}(v)=d_G(v) \le |C_i|-1$, and hence
\begin{equation}
    n-|C_i| \le d_{\overline{G}}(v)
\end{equation}
Combining (1) and (2), we obtain $|C_i| \ge 2k-1$. Since the choice of $v$ was arbitrary, this holds for every $i=1,\dots,r$. Now, let $T\subseteq V(G)$ with $|T|\le k-1$. Since each $|C_i|\ge 2k-1$, every set
$C_i\setminus T$ is nonempty. It follows that the spanning complete
$r$-partite subgraph of $\overline G-T$ with parts $C_1\setminus T,\dots,C_r\setminus T$ still has all parts nonempty, and therefore is connected. Hence $\overline G-T$ is connected. Therefore, $\overline G$ is $k$-connected. So, this case is not possible since we assumed that neither $G$ nor $\overline{G}$ is $k$-connected.

Next, consider the case where both $G$ and $\overline{G}$ are connected. Then, since
\[
\kappa(G)+\kappa(\overline{G}) \le (k-1)+(k-1)=2k-2
\]
and $\delta(G), \delta(\overline{G}) \ge 2k-2$, we have
\[
\min\{\delta(G), \delta(\overline{G})\}+1 \ge 2k-1
\]
which contradicts the theorem of Hellwig and Volkmann.

Thus, $\mathcal{C}(k)\le 2k-2$. Together with the lower bound, we conclude that $\mathcal{C}(k)=2k-2$.
\end{proof}

\section{An Application: A Nordhaus--Gaddum Type Product Condition}

Having determined the exact value of $\mathcal{C}(k)$, we now derive a Nordhaus–Gaddum type product condition as an application of our main result.

\begin{corollary}
Let $G$ be a graph on $n$ vertices and $k\ge 2$, where $n\ge 4k-5$. If $\delta(G)\delta(\overline{G}) \ge (2k-3)(n-2k+2)+1$, then either $G$ or $\overline{G}$ is $k$-connected.
\end{corollary}

\begin{proof} 
We prove the contrapositive. Suppose that neither $G$ nor $\overline{G}$ is $k$-connected. Then, by Theorem \ref{C(k)_main}, either $\delta(G) \le 2k-3$ or $\delta(\overline{G}) \le 2k-3$. Without loss of generality, assume that $\delta(G) \le 2k-3$. Since 
\[
\delta(\overline{G})=n-1-\Delta(G) \le n-1-\delta(G),
\] 
we have 
\[
\delta(G)\delta(\overline{G}) \le \delta(G)(n-1-\delta(G)),
\]
where the RHS is an increasing function of $\delta(G)$ when $\displaystyle \delta(G) \le \frac{n-1}{2}$. As $n\ge 4k-5$ and $\delta(G) \le 2k-3$, it follows that the RHS attains its maximum value when $\delta(G)=2k-3$, that is,  
\[
\delta(G)\delta(\overline{G}) \le \delta(G)(n-1-\delta(G)) \le (2k-3)(n-2k+2)
\]
and thus the result follows.
\end{proof}

Now, we focus on the equality case. From the argument used in the proof of the corollary, equality can occur only when $\delta(G)=2k-3$ and $\delta(\overline{G})=n-1-\delta(G)$. Since $\delta(\overline{G})=n-1-\Delta(G)$,
it follows that
\[
\delta(G)=\Delta(G)=2k-3.
\]
In other words, equality can occur only when $G$ is $(2k-3)$-regular.
The following structural result eliminates this equality case and
allows us to strengthen the bound found in  the previous corollary.

\begin{theorem} \label{regularity_theorem}
Let $k\ge2$, and let $G$ be a $(2k-3)$-regular graph on
$n\ge4k-5$ vertices. Then either $G$ or $\overline{G}$ is
$k$-connected.
\end{theorem}

\begin{proof}
First, note that if $n=4k-5$, then both $n$ and $2k-3$ are odd, so no $(2k-3)$-regular graph on $n$ vertices exists by the Handshaking Lemma. Hence, we may assume that $n\ge4k-4$.

Suppose that $G$ is not $k$-connected. Then there exists a set
$S_0\subseteq V(G)$ with $|S_0|\le k-1$ such that $G-S_0$ is
disconnected. For every component $C$ of $G-S_0$ and every
$v\in V(C)$, we have
\[
2k-3=d_G(v)\le |C|-1+|S_0|,
\]
and so
\[
|C|\ge 2k-2-|S_0|.
\]
Thus, $S_0$ may be enlarged to a set $S$ of size $k-1$ without deleting an entire component, and $G-S$ remains disconnected.

Let $C_1,\ldots,C_r$, where $r\ge2$, be the components of $G-S$.
For any $v\in V(C_i)$, all neighbors of $v$ in $G$ lie in
$V(C_i)\cup S$. Hence,
\[
2k-3=d_G(v)\le |C_i|-1+|S|=|C_i|-1+(k-1),
\]
and therefore $|C_i|\ge k-1$, for every $i=1,\ldots,r$.

We shall show that $\overline{G}$ is $k$-connected. Let
$T\subseteq V(G)$ be arbitrary with $|T|\le k-1$, and set
\[
U=V(G)\setminus(S\cup T)
 =\bigcup_{i=1}^{r}\bigl(V(C_i)\setminus T\bigr).
\]

First, suppose that at least two of the sets $V(C_i)\setminus T$
are nonempty. Since there are no edges in $G$ between distinct
components of $G-S$, the graph $\overline{G}[U]$ contains a spanning
complete multipartite graph whose nonempty parts are the sets
$V(C_i)\setminus T$. Hence, $\overline{G}[U]$ is connected.

It remains to connect the vertices of $S\setminus T$ to $U$. Let
$x\in S\setminus T$. If $x$ had no neighbor in $U$ in $\overline{G}$,
then $x$ would be adjacent to every vertex of $U$ in $G$. But this means
\[
d_G(x)
\ge |U|
= n-|S|-|T|+|S\cap T|
\ge n-|S|-|T|
\ge n-2k+2
\ge 2k-2,
\]
whereas $d_G(x)=2k-3$, a contradiction. Therefore, every vertex of
$S\setminus T$ has a neighbor in $U$ in $\overline{G}$, and hence
$\overline{G}-T$ is connected.

It remains to consider the case in which at most one of the sets
$V(C_i)\setminus T$ is nonempty. Since $r\ge2$, the set $T$ must
contain all vertices of at least one component, say $C_j$. Since
$|C_j|\ge k-1$ and $|T|\le k-1$, we obtain
\[
|C_j|=|T|=k-1
\qquad\text{and}\qquad
T=V(C_j).
\]
Moreover, $T$ cannot contain all vertices of two distinct components,
since each component has at least $k-1$ vertices while $|T|=k-1$.
Since at most one of the sets $V(C_i)\setminus T$ is nonempty, we obtain that $r=2$. Denote these
components by $A$ and $B$, where
\[
|A|=|S|=|T|=k-1
\qquad\text{and}\qquad
T=A.
\]

Let $a\in A$. Since $A$ is a component of $G-S$, $a$ is not adjacent to any vertex of $B$ in $G$. Thus, all neighbors of $a$ lie in
$(A\setminus\{a\})\cup S$,
which contains exactly $(k-2)+(k-1)=2k-3$ vertices. Since $d_G(a)=2k-3$, $a$ is adjacent to every
vertex of this set. Consequently, $G[A]$ is complete and every
vertex of $A$ is adjacent to every vertex of $S$.

We now show that $\overline{G}-A=\overline{G}[S\cup B]$
is connected. Suppose for a contradiction that it is disconnected.
For each $s\in S$, let $Q_s$ be the component of
$\overline{G}[S\cup B]$ containing $s$. Since $s$ is adjacent to all $k-1$ vertices of $A$ in $G$, it has at most $(2k-3)-(k-1)=k-2$ neighbors in $S\cup B$. Moreover, every vertex of
$(S\cup B)\setminus Q_s$ is adjacent to $s$ in $G$, since there are
no edges of $\overline{G}$ between distinct components. Therefore,
\[
|(S\cup B)\setminus Q_s|\le k-2,
\]
and hence
\[
|Q_s|
\ge |S\cup B|-(k-2)
= n-2k+3
\ge 2k-1.
\]

Here, we claim that all vertices of $S$ lie in the same component of
$\overline{G}[S\cup B]$. Suppose that $s,t\in S$ lie in distinct
components $Q_s$ and $Q_t$. Since
$Q_t\subseteq (S\cup B)\setminus Q_s$,
we have $|Q_t|\le k-2$. On the other hand, applying the preceding
bound to $t$ gives $|Q_t|\ge2k-1$, which is a contradiction establishing the claim. Thus, all
vertices of $S$ belong to a single component, say $Q$.

Since $\overline{G}[S\cup B]$ is disconnected, the set $R=(S\cup B)\setminus Q$ is nonempty. Since $S\subseteq Q$, we have $R\subseteq B$. For any
$s\in S$, every vertex of $R$ is adjacent to $s$ in $G$. Moreover, we know that
$s$ is adjacent to every vertex of $A$ in $G$. Therefore,
\[
|A|+|R| = k-1+|R| \le d_G(s)=2k-3,
\]
which yields 
\[
|R|\le k-2.
\]

Now choose $b\in R$. Since $b$ and the vertices of $Q$ lie in
distinct components of $\overline{G}[S\cup B]$, the vertex $b$ is
adjacent to every vertex of $Q$ in $G$. Hence
\[
|Q| \le d_G(b) = 2k-3.
\]
However,
\[
|Q|
=|S\cup B|-|R|
=n-(k-1)-|R|
\ge n-2k+3
\ge 2k-1,
\]
which is a contradiction.

Thus, $\overline{G}-T$ is connected for every
$T\subseteq V(G)$ with $|T|\le k-1$. Therefore,
$\overline{G}$ is $k$-connected, and the result follows.
\end{proof}

The following corollary is immediate.

\begin{corollary} \label{main_application} Let $G$ be a graph on $n$ vertices and $k\ge 2$, where $n\ge 4k-5$. If $\delta(G)\delta(\overline{G}) \ge (2k-3)(n-2k+2)$, then either $G$ or $\overline{G}$ is $k$-connected.
\end{corollary}

Here, a natural question to ask is whether we can make the bound on $\delta(G)\delta(\overline{G})$ independent of $n$ and, ideally, dependent only on $k$. The answer to this question is negative as the following construction demonstrates.

\begin{example} Given $k\ge 3$, let $G=(V,E)$ such that $V=A \sqcup B\sqcup S$ where 
\[
A=\{a_1,\dots,a_r\},\, B=\{b_1,\dots,b_r\},\, S=\{s_1,\dots,s_{k-1}\}
\]
with $r\ge k$ and 
\[
E(G)=\{a_ia_j \mid 1\le i<j \le r\} \cup \{b_ib_j \mid 1\le i<j \le r\} \cup E(S,A) \cup E(S,B)
\]
Then,
\[
E(\overline{G}) = \{s_is_j \mid 1\le i < j \le k-1\} \cup E(A,B)
\]
One can see the construction of $G$ and $\overline{G}$ in the figure below.

\begin{figure}[htbp]
    \centering
    \includegraphics[width=15.5cm]{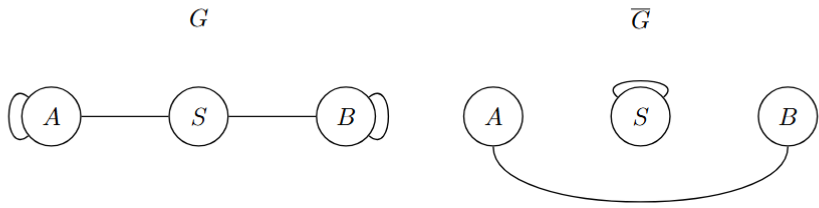}
    \caption{The construction of graphs $G$ and $\overline{G}$ that prove the bound on $\delta(G)\delta(\overline{G})$ cannot be dependent only on $k$}
\end{figure}

So, we have $\delta(G)=r+k-2$ and $\delta(\overline{G}) = k-2$ because $r\ge k$ implies $2r > r+k-2$ and $r > k-2$. Then,
\[
\delta(G)\delta(\overline{G}) = (r+k-2)(k-2)
\]
and thus, we can make the product $\delta(G)\delta(\overline{G})$ arbitrarily large by choosing $r$ large enough. However, neither $G$ nor $\overline{G}$ is $k$-connected since $G-S$ is disconnected and $\overline{G}$ itself is disconnected. Therefore, any bound on $\delta(G)\delta(\overline{G})$ dependent only on $k$ does not guarantee that either $G$ or $\overline{G}$ is $k$-connected.
\end{example}

\noindent\rule{\textwidth}{0.4pt}

\section*{Acknowledgements}

I would like to express my gratitude to the late Walter Kern, who first led me to the complement connectivity problem in 2019, when, as an Erasmus student, I took his Graph Theory course and he asked me about the case $k=2$. The proof of Theorem~\ref{thm_k2} is the same argument that I presented to him at the time.

\section*{Tool and Computational Resource Disclosure}

The author used OpenAI's ChatGPT (GPT-5.6) for language editing and for improving the presentation and exposition of the manuscript. ChatGPT was also used to assist in developing the argument showing that $\overline{G}-A$ is connected in the proof of Theorem~\ref{regularity_theorem}. The mathematical argument and all AI-assisted text were reviewed, verified, and incorporated into the manuscript by the author, who takes full responsibility for the content.

\bibliographystyle{unsrt}
\bibliography{references}

\end{document}